\documentclass[twoside,leqno,10pt]{amsart}
\usepackage{amsfonts}
\usepackage{amsmath}
\usepackage{amscd}
\usepackage{amssymb}
\usepackage{amsthm}
\usepackage{amsrefs}
\usepackage{latexsym}
\usepackage{mathrsfs}
\usepackage{bbm}
\usepackage{enumerate}
\usepackage{color}
\begin{document}

\newtheorem{theorem}[subsection]{Theorem}
\newtheorem{proposition}[subsection]{Proposition}
\newtheorem{lemma}[subsection]{Lemma}
\newtheorem{corollary}[subsection]{Corollary}
\newtheorem{conjecture}[subsection]{Conjecture}
\newtheorem{prop}[subsection]{Proposition}
\numberwithin{equation}{section}
\newcommand{\mr}{\ensuremath{\mathbb R}}
\newcommand{\mc}{\ensuremath{\mathbb C}}
\newcommand{\dif}{\mathrm{d}}
\newcommand{\intz}{\mathbb{Z}}
\newcommand{\ratq}{\mathbb{Q}}
\newcommand{\natn}{\mathbb{N}}
\newcommand{\comc}{\mathbb{C}}
\newcommand{\rear}{\mathbb{R}}
\newcommand{\prip}{\mathbb{P}}
\newcommand{\uph}{\mathbb{H}}
\newcommand{\fief}{\mathbb{F}}
\newcommand{\majorarc}{\mathfrak{M}}
\newcommand{\minorarc}{\mathfrak{m}}
\newcommand{\sings}{\mathfrak{S}}
\newcommand{\fA}{\ensuremath{\mathfrak A}}
\newcommand{\mn}{\ensuremath{\mathbb N}}
\newcommand{\mq}{\ensuremath{\mathbb Q}}
\newcommand{\half}{\tfrac{1}{2}}
\newcommand{\f}{f\times \chi}
\newcommand{\summ}{\mathop{{\sum}^{\star}}}
\newcommand{\chiq}{\chi \bmod q}
\newcommand{\chidb}{\chi \bmod db}
\newcommand{\chid}{\chi \bmod d}
\newcommand{\sym}{\text{sym}^2}
\newcommand{\hhalf}{\tfrac{1}{2}}
\newcommand{\sumstar}{\sideset{}{^*}\sum}
\newcommand{\sumprime}{\sideset{}{'}\sum}
\newcommand{\sumprimeprime}{\sideset{}{''}\sum}
\newcommand{\shortmod}{\ensuremath{\negthickspace \negthickspace \negthickspace \pmod}}
\newcommand{\V}{V\left(\frac{nm}{q^2}\right)}
\newcommand{\sumi}{\mathop{{\sum}^{\dagger}}}
\newcommand{\mz}{\ensuremath{\mathbb Z}}
\newcommand{\leg}[2]{\left(\frac{#1}{#2}\right)}
\newcommand{\muK}{\mu_{\omega}}
\newcommand{\sumflat}{\sideset{}{^\flat}\sum}
\newcommand{\thalf}{\tfrac12}
\newcommand{\lam}{\lambda}

\title[The mean square of real character sums]{The mean square of real character sums}

\date{\today}
\author{Peng Gao}
\address{Department of Mathematics, School of Mathematics and Systems Science, Beihang University, P. R. China}
\email{penggao@buaa.edu.cn}

\begin{abstract}
In this paper, we evaluate a smoothed sum of the form $\displaystyle \sumstar_{d \leq X}\left(\sum_{n \leq Y} \leg {8d}{n} \right)^2$, where $\leg {8d}{\cdot}$ is the Kronecker symbol and $\sumstar$ denotes a sum over positive odd square-free integers.
\end{abstract}

\maketitle


\section{Introduction}

  Estimations for character sums play vital roles in analytic number theory. Among the many results developed on this topic, we recall the classical P\'olya-Vinogradov inequality (see for example \cite[Chap. 23]{Da}), which asserts that for any non-principal Dirichlet character $\chi$ modulo $q$  and $M \in \mz$, $N \in \mn$, we have
\begin{align}
\label{pv}
   \sum_{M < n \leq M+N} \chi(n) \ll q^{1/2}\log q.
\end{align}

  The P\'olya-Vinogradov inequality can be regarded as a first moment estimation for Dirichlet characters and we may thus consider estimations for higher powers of character sums. For this, we restrict our attention to the second moment (the mean square) and the set of real Dirichlet characters, i.e. quadratic or principal characters. In this setting, we note the following result due to M. V. Armon \cite[Theorem 2]{Armon}:
\begin{align}
\label{squareJacobi}
   \sum_{\substack {|D| \leq X \\ D \in \mathcal{D}}} \left| \sum_{n \leq Y} \Big (\frac {D}{n} \Big ) \right|^2 \ll XY \log X,
\end{align}
   where $\mathcal{D}$ is the set of non-square quadratic discriminants and $\chi_D=(\frac {D}{\cdot})$ is the Kronecker symbol. The above estimation was initiated by M. Jutila in \cites{Jutila1}, who showed the result useful by applying weaker estimations in \cites{Jutila, Jutila1} to study problems related to the mean values of class numbers of imaginary quadratic number  fields and the second moment of Dirichlet $L$-functions with primitive
   quadratic characters.

   We remark that it follows from the proof of \cite[Theorem 2]{Armon} that when $Y \leq X^{1/2}$, we can obtain an asymptotic formula for the mean square expression considered in \eqref{squareJacobi} such that for some constant $c$,
\begin{align}
\label{squareJacobiasmp}
   \sum_{\substack {|D| \leq X \\ D \in \mathcal{D}}} \left| \sum_{n \leq Y} \Big (\frac {D}{n} \Big ) \right|^2 =c XY \log Y+o(XY\log Y)+O(Y^3 \log Y).
\end{align}

   In view of this and the P\'olya-Vinogradov inequality, one may attempt to obtain an asymptotic formula for the sum in \eqref{pv} with $\chi(n)$ replaced by $\chi_D(n)$ there, provided that one takes an extra average over $D$ as well. More explicitly, one can consider the following sum:
\begin{align*}
    \sum_{\substack{m \leq Y \\ (m,2)=1}}\sum_{\substack{n \leq Y \\ (n,2)=1}} \Big (\frac {m}{n} \Big ).
\end{align*}

    In \cite{CFS},  J. B. Conrey, D. W. Farmer, and K. Soundararajan obtained an asymptotic expression for the above sum for all $X,Y>0$. We note here that the most subtle situation is when $X$ and $Y$ are of comparable size since it is relatively easy to apply the P\'olya-Vinogradov inequality to control the error terms when $X, Y$ are far apart in size. A key step used in \cite{CFS} is an application of a type of Poisson summation formula
(see Lemma \ref{lem2} below) involving quadratic Dirichlet characters established by Soundararajan in \cite[Lemma 2.6]{sound1}.

   The method in \cite{CFS} can be applied to study mean values of any arithmetic function twisted by quadratic Dirichlet characters. For example, squaring out the left-hand side expression in \eqref{squareJacobi} and interchanging the order of
   summations yields
\begin{align}
\label{ms1}
   \sum_{\substack {|D| \leq X \\ D \in \mathcal{D}}} \left| \sum_{n \leq Y} \Big (\frac {D}{n} \Big ) \right|^2 =\sum_{n \leq Y}\sum_{\substack {|D| \leq
   X \\ D \in \mathcal{D}}} \leg{D}{n}d(n)+\sum_{\substack{n > Y \\ n=n_1n_2\\ n_1 \leq Y, n_2 \leq Y}}\sum_{\substack {|D| \leq X \\ D \in \mathcal{D}}}
   \leg{D}{n},
\end{align}
   where $d(n)$ is the divisor function.

   Consideration on the first term of the right-hand side expression above leads to the study of the following sum:
\begin{align}
\label{quaddiv}
  \sum_{\substack {m \leq X \\ (m, 2)=1}}\sum_{\substack {n \leq Y \\ (n, 2)=1}} \leg {m}{n}d(n).
\end{align}
  Following the treatment in \cite{CFS}, an asymptotic formula for the above sum is given by the author in \cite{Gao1} recently, which is valid for all large $X, Y$ satisfying $X^{\epsilon} \ll Y \ll X^{1-\epsilon}$ for any $\epsilon>0$.

   Comparing \eqref{ms1} with  \eqref{quaddiv}, we see that \eqref{quaddiv} resembles part of the left-hand side expression in \eqref{squareJacobiasmp}. It is certainly interesting then to study the right-hand side expression in \eqref{squareJacobiasmp} itself, with the aim to extend the asymptotic formula given in \eqref{squareJacobiasmp} to $Y>X^{1/2}$. It is our goal in this paper to do so and for technical reasons, we consider the following sum
\begin{align}
\label{SXY}
    S(X,Y;\Phi, W)=\sumstar_{\substack {d}} \left| \sum_{n} \Big (\frac {8d}{n} \Big )\Phi \left(\frac {n}{Y} \right ) \right |^2 W\left(\frac {d}{X} \right ),
\end{align}
   where $\Phi$ and $W$ are smooth, compactly supported functions and we use $\sumstar$ to denote a sum over positive odd square-free
   integers throughout the paper.

   Similar to \eqref{squareJacobiasmp}, it is easy to obtain an asymptotic formula for $S(X,Y;\Phi, W)$ when $Y \leq X^{1/2}$. Thus, we may assume that $Y^2>X$ and our result in this paper is the following
\begin{theorem}
\label{meansquare}
   Let $\Phi$ and $W$ be two smooth, compactly supported functions whose supports are contained in $(0,1)$. Let $S(X,Y;\Phi, W)$ be defined in \eqref{SXY}. For large $X$ and  $Y$ with $Y \leq X \leq Y^2$, we have for any $\epsilon>0$,
\begin{align}
\label{S}
\begin{split}
   S(X,Y;\Phi, W) =& C_1(\Phi, W)XY\log Y +C_2(\Phi, W)XY\log (\frac {Y^2}{X})+O \left (XY+XY(\frac {Y^2}{X} )^{-\frac 18+\epsilon}+(XY)^{1 + \varepsilon}\left (\frac YX \right)^{1/2} \right),
\end{split}
\end{align}
  where $C_1(\Phi, W), C_2(\Phi, W)$ are given in \eqref{C1C2}.
\end{theorem}

   One checks that \eqref{S} gives a valid asymptotic formula when $X^{1/2} \ll Y \ll X^{1-\epsilon}$ for any $\epsilon>0$. The proof of Theorem \ref{meansquare} uses not only ideas of \cite{CFS}, but also those from \cite{Young1, Young2} and \cite{S&Y}. In particular, K. Soundararajan and M.P. Young used in \cite{S&Y} a triple contour integral to evaluate asymptotically a smoothed sum of products of coefficients of Fourier expansions of modular forms twisted by quadratic Dirichlet characters. Our approach in the proof of Theorem \ref{meansquare} is clearly inspired by their method.

\section{Preliminaries}
\label{sec 2}

   In this section, we gather a few auxiliary results needed in the proof of Theorem \ref{meansquare}.
\subsection{Gauss sums}
\label{section:Gauss}

   For all odd
    integers $k$ and all integers $m$, we introduce the following Gauss-type
    sums as in \cite[Sect. 2.2]{sound1}
\begin{align}
\label{G}
    G_m(k)=
    \left( \frac {1-i}{2}+\left( \frac {-1}{k} \right)\frac {1+i}{2}\right)\sum_{a \shortmod{k}}\left( \frac {a}{k} \right) e \left( \frac {am}{k} \right),
\end{align}
   where $e(x)=e^{2\pi i x}$. We quote \cite[Lemma 2.3]{sound1} which determines $G_m(k)$.
\begin{lemma}
\label{lem1}
   If $(k_1,k_2)=1$ then $G_m(k_1k_2)=G_m(k_1)G_m(k_2)$. Suppose that $p^a$ is
   the largest power of $p$ dividing $m$ (put $a=\infty$ if $m=0$).
   Then for $b \geq 1$ we have
\begin{equation*}
\label{011}
    G_m(p^b)= \left\{\begin{array}{cl}
    0  & \mbox{if $b\leq a$ is odd}, \\
    \varphi(p^b) & \mbox{if $b\leq a$ is even},  \\
    -p^a  & \mbox{if $b=a+1$ is even}, \\
    (\frac {m/p^a}{p})p^a\sqrt{p}  & \mbox{if $b=a+1$ is odd}, \\
    0  & \mbox{if $b \geq a+2$}.
    \end{array}\right.
\end{equation*}
\end{lemma}

\subsection{Poisson Summation}
   For a Schwartz function $F$, we define
\begin{equation} \label{tildedef}
   \widetilde{F}(\xi)=\frac {1+i}{2}\hat{F}(\xi)+\frac
   {1-i}{2}\hat{F}(-\xi)=\int\limits^{\infty}_{-\infty}\left(\cos(2\pi \xi
   x)+\sin(2\pi \xi x) \right)F(x) \dif x,
\end{equation}
   where $\hat{F}$ denotes the Fourier transform of $F$.

    We have the following Poisson summation formula from \cite[Lemma 2.6]{sound1}:
\begin{lemma}
\label{lem2}
   Let $W$ be a smooth function with compact support on the
positive real numbers. For any odd integer $n$,
\begin{equation*}
\label{013}
  \sum_{(d,2)=1}\left( \frac {d}{n} \right)
    W\left( \frac {d}{X} \right)=\frac {X}{2n}\left( \frac {2}{n} \right)
    \sum_k(-1)^kG_k(n)\widetilde{W}\left( \frac {kX}{2n} \right),
\end{equation*}
where $\widetilde{W}$ is defined in \eqref{tildedef} and $G_k(n)$ is defined in \eqref{G}.
\end{lemma}

\subsection{A mean value estimate}
   In the proof of Theorem \ref{meansquare}, we need the following mean value estimate for Dirichlet $L$-functions.
\begin{lemma} \cite[Lemma 2.5]{sound1}
\label{lem:HB}
 Let $S(Q)$ denote the set of real, primitive characters $\chi$ with conductor $\leq Q$. For any complex number $\sigma+it$  with $\sigma \geq \frac 12$, we have
\begin{equation*}
\label{eq:H-B}
 \sumstar_{\substack{\chi \in S(Q)}} |L(\sigma + it, \chi)|^2 \ll_{\varepsilon} (X(1+|t|)^{1/2})^{1 + \varepsilon}.
\end{equation*}
\end{lemma}

\subsection{Analytical behaviors of some Dirichlet Series}
   Let $G_m(k)$ be defined as in \eqref{G}. Let $\epsilon \in \{ \pm \}$ and $k_1$ be square-free, we define
\begin{equation}
\label{eq:Z}
 Z_{\epsilon}(\alpha,\beta,\gamma;q,k_1) = \sum_{k_2=1}^{\infty} \sum_{(n_1,2q)=1} \sum_{(n_2,2q)=1} \frac{1}{n_1^{\alpha} n_2^{\beta} k_2^{2\gamma}} \frac{G_{\epsilon k_1 k_2^2}(n_1 n_2)}{n_1 n_2}.
\end{equation}
 We note that $Z_{\epsilon}(\alpha,\beta,\gamma;q,k_1)$ converges absolutely if Re$(\alpha)$, Re$(\beta)$, and Re$(\gamma)$ are
all $> \frac 12$.  In fact, we have the following analytical behavior of $Z_{\epsilon}$.
\begin{lemma}
\label{lemma:Z}
 The function $Z_{\epsilon}(\alpha,\beta,\gamma;q,k_1)$ defined above may be written
 as
 \begin{align*}
 L(\frac 12+\alpha, \chi_{\epsilon k_1})L(\frac 12+\beta,\chi_{\epsilon k_1})Z_{2, \epsilon}(\alpha,\beta,\gamma;q, k_1),
\end{align*}
 where $Z_{2, \epsilon}(\alpha,\beta,\gamma;q,k_1)$ is a function uniformly bounded by $(qk_1)^{\epsilon}$ for any $\epsilon>0$ in the region
$\text{\rm Re}(\gamma) \ge \frac 12+\varepsilon$, and $\text{\rm Re}(\alpha), \text{\rm Re}(\beta) \geq \epsilon$.
\end{lemma}
\begin{proof}
It follows from Lemma \ref{lem1} that the summand of \eqref{eq:Z} is jointly multiplicative in terms of $n_1, n_2$, and $k_2$, so that $Z_{\epsilon}(\alpha,\beta,\gamma;q, k_1)$ can be
expressed as an Euler product over all primes $p$ with each Euler factor at $p$ being
\begin{align*}
 Z_{\epsilon,p}(\alpha,\beta,\gamma;q, k_1):=\sum_{k_2, n_1, n_2} \frac{1}{p^{n_1 \alpha + n_2 \beta +2k_2 \gamma}}\frac{ G_{\epsilon k_1 p^{2k_2}}(p^{n_1 + n_2})}{p^{n_1+n_2}  }.
\end{align*}
 Note that we have $ G_{\epsilon k_1 p^{2k_2}}(p^{n_1 + n_2}) \leq p^{n_1+n_2}$ by Lemma \ref{lem1}, it follows that we have $ Z_{\epsilon,p}(\alpha,\beta,\gamma;q, k_1) \ll 1$ uniformly for all $p$.

 To analyze $Z_{\epsilon,p}$, we consider the generic case when
$p \nmid 2q k_1$.  We first evaluate $G_{\epsilon k_1 p^{2k_2}}(p^{n_1 + n_2})$ explicitly using Lemma \ref{lem1}, then upon replacing $G_{\epsilon k_1 p^{2k_2}}(p^{n_1 + n_2})$ by these explicit expressions in the definition of $Z_{\epsilon,p}$ and keeping only the non-zero terms, we obtain an alternative expression for $Z_{\epsilon,p}$, and we denote this expression by $Z^{gen}_{\epsilon,p}(\alpha,\beta,\gamma;q, k_1)$. One checks that we have
\begin{equation*}
 Z^{gen}_{\epsilon,p}(\alpha,\beta,\gamma;q, k_1):=\sum_{k_2 \geq 0} \left ( \sum_{\substack{ n_1, n_2 \geq 0 \\ n_1+n_2 \equiv 0 \pmod 2 \\ n_1+n_2 \leq 2k_2}} \frac{1}{p^{n_1 \alpha + n_2 \beta +2k_2 \gamma}}\frac{ \varphi(p^{n_1 + n_2})}{p^{n_1+n_2}}+\sum_{\substack{ n_1, n_2 \geq 0 \\ n_1+n_2 = 2k_2+1}} \frac{1}{p^{n_1 \alpha + n_2 \beta +2k_2 \gamma}}\frac{\chi_{\epsilon k_1}(p)}{p^{1/2}  }  \right ).
\end{equation*}
 We now extend the above definition of $Z^{gen}_{\epsilon,p}(\alpha,\beta,\gamma;q, k_1)$ to all other $p$.

In the region Re$(\gamma) \ge \frac 12+\varepsilon$, Re$(\alpha)$, Re$(\beta)\ge \epsilon$, it follows from
the definition of $Z^{gen}_{\epsilon,p}$ that the contribution from terms $k_2\ge 1$ is of size
$\ll 1/p^{1+2\varepsilon}$. The contribution of the
term $k_2=0$ is $1+ \chi_{\epsilon k_1}(p) (p^{-\frac 12-\alpha}+p^{-\frac 12-\beta})$.

  We now define
\begin{align*}
  Z_{2, \epsilon}(\alpha,\beta,\gamma;q, k_1) =& \left ( L(\frac 12+\alpha, \chi_{\epsilon k_1})L(\frac 12+\beta,\chi_{\epsilon k_1})\right )^{-1}Z_{\epsilon}(\alpha,\beta,\gamma;q,k_1) \\
 =& Z^{gen}_{2, \epsilon}(\alpha,\beta,\gamma;q, k_1)Z^{non-gen}_{2, \epsilon}(\alpha,\beta,\gamma;q, k_1),
\end{align*}
  where
\begin{align*}
  Z^{gen}_{2, \epsilon}(\alpha,\beta,\gamma;q, k_1)=& \prod_p(1-\frac {\chi_{\epsilon k_1}(p)}{p^{\frac 12+\alpha}})(1-\frac {\chi_{\epsilon k_1}(p)}{p^{\frac 12+\beta}})Z^{gen}_{\epsilon,p}(\alpha,\beta,\gamma;q, k_1), \\
  Z^{non-gen}_{2, \epsilon}(\alpha,\beta,\gamma;q, k_1)=& \prod_{p | 2qk_1}(1-\frac {\chi_{\epsilon k_1}(p)}{p^{\frac 12+\alpha}})(1-\frac {\chi_{\epsilon k_1}(p)}{p^{\frac 12+\beta}})Z^{gen}_{\epsilon,p}(\alpha,\beta,\gamma;q, k_1)^{-1}Z_{\epsilon,p}(\alpha,\beta,\gamma;q, k_1).
\end{align*}
  Our arguments above imply that $Z^{gen}_{2, \epsilon}$ is uniformly bounded by $1$ in the region
$\text{\rm Re}(\gamma) \ge \frac 12+\varepsilon$, and $\text{\rm Re}(\alpha), \text{\rm Re}(\beta) \geq \epsilon$. As one checks easily that $Z^{non-gen}_{2, \epsilon}$ is uniformly bounded by $(qk_1)^{\epsilon}$ for any $\epsilon>0$ in the same region, the assertions of the lemma now follow.
\end{proof}

    Our next lemma concerns with the analytical behavior of $Z_{2, 1}$.
\begin{lemma}
\label{lemma:Z2}
   The function $Z_{2, 1}(\frac 12, \frac 12,\gamma;q, 1)$ defined in Lemma \ref{lemma:Z} may be written
 as
 \begin{align*}
 \zeta(2\gamma)\zeta(1+2\gamma)Z_{3}(\gamma;q),
\end{align*}
where $Z_{3}(\gamma;q)$ converges uniformly in the region
$\text{\rm Re}(\gamma) \ge -\frac 18+\varepsilon$ and satisfies $Z_{3}(\gamma;q), Z^{(i)}_{3}(0;q) \ll q^{\epsilon}$ for any $\epsilon>0$, $0 \leq i \leq 2$ in the same region.

  Let Re$(\alpha)=\epsilon>0$. The function $Z_{2, 1}(\alpha,\frac 12,\gamma;q, 1)$ defined in Lemma \ref{lemma:Z} may be written as
\begin{align}
 \label{Z4}
 \zeta(2\gamma)\zeta(2(\gamma+\alpha))Z_{4}(\gamma; \alpha, q),
\end{align}
where $Z_{4}(\gamma;\alpha, q)$ is a function uniformly bounded by $q^{\epsilon}$ for any $\epsilon>0$ in the region
$\text{\rm Re}(\gamma) \ge \frac 14+\varepsilon$.
\end{lemma}
\begin{proof}
  It follows from the proof of Lemma \ref{lemma:Z} that for either $\alpha=\frac 12$ or Re$(\alpha)=\epsilon>0$, $Z_{2,1}(\alpha, \frac 12,\gamma;q, 1)$ has an Euler product over all primes $p$ in the region
$\text{\rm Re}(\gamma) \ge \frac 12+\varepsilon$, with each Euler factor at $p$ being
\begin{align}
\label{Z21p}
 Z_{2,1,p}(\alpha, \frac 12,\gamma;q, 1):=(1-\frac 1{p^{1/2+\alpha}})(1-\frac 1{p})\sum_{k_2, n_1, n_2} \frac{1}{p^{n_1\alpha+ n_2/2 +2k_2 \gamma}}\frac{ G_{p^{2k_2}}(p^{n_1 + n_2})}{p^{n_1+n_2}  }.
\end{align}
  Note that we have $ G_{p^{2k_2}}(p^{n_1 + n_2}) \leq p^{n_1+n_2}$ by Lemma \ref{lem1}, it follows that we have $Z_{2,1,p}(\alpha, \frac 12,\gamma;q, 1) \ll 1$ uniformly for all $p$.

 We first consider the case $\alpha=\frac 12$. To analyze $Z_{2,1,p}$, we consider the generic case when
$p \nmid 2q$.  We first evaluate $G_{p^{2k_2}}(p^{n_1 + n_2})$ explicitly using Lemma \ref{lem1}, then upon replacing $G_{ p^{2k_2}}(p^{n_1 + n_2})$ by these explicit expressions in the definition of $Z_{2,1,p}$ and keeping only the non-zero terms, we obtain an alternative expression for $Z_{2,1,p}$, and we denote this expression by $Z^{gen}_{2,1,p}(\frac 12, \frac 12,\gamma;q, 1)$. One checks that
\begin{align*}
   & (1-\frac 1{p})^{-2}Z^{gen}_{2,1,p}(\frac 12, \frac 12,\gamma;q, 1)\\
   :=&\sum_{\substack{n_1+n_2=i \\n_1, n_2 \geq 0 }} \sum_{k_2 \geq 0} \frac{1}{p^{i/2 +2k_2 \gamma}}\frac{ G_{p^{2k_2}}(p^{i})}{p^{i}  } \\
   =& \frac {1}{1-p^{-2\gamma}}+\frac 2{p}+(1-\frac 1p)\frac {3p^{-1-2\gamma}}{1-p^{-2\gamma}}+\sum_{\substack{n_1+n_2=i \geq 3 \\n_1, n_2 \geq 0 }} \sum_{k_2 \geq 1} \frac{1}{p^{i/2 +2k_2 \gamma}}\frac{ G_{p^{2k_2}}(p^{i})}{p^{i}  } \\
    =& \frac {1}{1-p^{-2\gamma}}+\frac 2{p}+(1-\frac 1p)\frac {3p^{-1-2\gamma}}{1-p^{-2\gamma}}+(1-\frac 1p)\sum_{j \geq 2} \frac {(2j+1)p^{-j-2j\gamma}}{1-p^{-2\gamma}}+(1-\frac 1p)\sum_{l \geq 1} \frac {(2l+2)p^{-l-1/2-2(l+1)\gamma}}{1-p^{-2\gamma}}+\sum_{l' \geq 1} \frac {2l'+2}{p^{l'+1+2l'\gamma}} \\
   =& \frac {1}{1-p^{-2\gamma}}+\frac 2{p}+(1-\frac 1p)\frac {3p^{-1-2\gamma}}{1-p^{-2\gamma}}+O(\frac {p^{-3/2-4\gamma+\epsilon}+p^{-2-2\gamma+\epsilon}}{1-p^{-2\gamma}}).
\end{align*}
 We now extend the above definition of $Z^{gen}_{2,1,p}(\frac 12, \frac 12,\gamma;q, 1)$ to all other $p$.

  We now define
\begin{align*}
  Z_{3}(\gamma;q) =& \left ( \zeta(2\gamma)\zeta(1+2\gamma) \right )^{-1}Z_{2, 1}(\frac 12, \frac 12,\gamma;q, 1) \\
 =& Z^{gen}_{3}(\gamma;q)Z^{non-gen}_{3}(\gamma;q),
\end{align*}
  where
\begin{align*}
  Z^{gen}_{3}(\gamma;q)=& \prod_p(1-\frac {1}{p^{2\gamma}})(1-\frac {1}{p^{1+2\gamma}})Z^{gen}_{2,1,p}(\frac 12, \frac 12,\gamma;q, 1), \\
  Z^{non-gen}_{3}(\gamma;q)=& \prod_{p | 2q}(1-\frac {1}{p^{2\gamma}})(1-\frac {1}{p^{1+2\gamma}})Z^{gen}_{2,1,p}(\frac 12, \frac 12,\gamma;q, 1)^{-1}Z_{2,1,p}(\frac 12, \frac 12,\gamma;q, 1).
\end{align*}
  Our arguments above imply that $Z^{gen}_{3}$ is uniformly bounded by $1$ in the region
$\text{\rm Re}(\gamma) \ge -\frac 18+\varepsilon$. One checks easily that $Z^{non-gen}_{3}$ is uniformly bounded by $q^{\epsilon}$ for any $\epsilon>0$ in the same region, and the first two derivatives of $Z^{non-gen}_{3}$ at $\gamma=0$ satisfy the same bound. This proves the first part of the lemma.

   As the proof for the second part of the lemma is similar, we shall only sketch a proof for \eqref{Z4} here. For this, it suffices to analyze the contribution of an Euler factor given in \eqref{Z21p} for the generic case when $p \nmid 2q$.
  In the region Re$(\gamma) \ge \frac 14+\varepsilon$, Re$(\alpha)$, Re$(\beta)\ge \epsilon$, it follows from
Lemma \ref{lem1} that the contribution from terms $k_2\ge 2$ is of size
$\ll 1/p^{1+2\varepsilon}$.  It also follows from the proof of Lemma \ref{lemma:Z} that the contribution of the
term $k_2=0$ is  $1+ O(1/p^{1+2\varepsilon})$. This leaves the contribution of the
term $k_2=1$ which is non-zero only when $n_1+n_2=0, 2$ or $3$. It is ready to check that the contributions of the terms $n_2 \geq 1$ in these cases are $\ll 1/p^{1+2\varepsilon}$. We are thus led to consider only the cases $n_1=n_2=0, k_2=1$ or $n_1=2, n_2=0, k_2=1$ or $n_1=3,  n_2=0, k_2=1$. Using Lemma \ref{lem1}, we see that these terms contribute $\frac 1{p^{2\gamma}}+\frac 1{p^{2\alpha+2\gamma}}+O(\frac {1}{p^{1+2\varepsilon}})$ and \eqref{Z4} follows from this easily.

\end{proof}

\section{Proof of Theorem \ref{meansquare}}
\label{sec 3}

\subsection{Decomposition of $S(X,Y;\Phi, W)$}
\label{section:mainprop}

   We first recast $S(X,Y;\Phi, W)$ as
\[
S(h):= \sumstar_{d} \sum_{n_1} \sum_{n_2} \chi_{8d}(n_1n_2) h(d,n_1,n_2),
\]
where $h(x,y,z)=W(\frac xX)\Phi \left(\frac {y}{Y} \right )\Phi \left(\frac {z}{Y} \right )$ is a smooth function on ${\mr}_+^3$.
We now apply the M\"{o}bius inversion to remove the square-free condition on $d$. Thus we obtain, for an appropriate parameter $Z$ to be chosen later,
\begin{eqnarray*}
 S(h) &=& \Big(\sum_{\substack{a \leq Z \\ (a,2)=1}} + \sum_{\substack{a > Z \\ (a,2)=1}} \Big) \mu(a)  \sum_{(d,2)=1} \sum_{(n_1,a)=1} \sum_{(n_2,a)=1} \chi_{8d}(n_1n_2)h(da^2, n_1, n_2)\\
 &=&S_1(h)+ S_2(h).
\end{eqnarray*}

\subsection{Estimating $S_2(h)$}
\label{section:S2}
  In this section, we estimate $S_2(h)$. We write $d=b^2 \ell$ with $\ell$ being square-free, and
group terms according to $c=ab$ to see that
\begin{equation}
\label{eq:S21}
 S_2(h) = \sum_{(c,2)=1} \sum_{\substack{a > Z \\ a|c}} \mu(a)
 \sumstar_{\ell} \sum_{(n_1,c)=1} \sum_{(n_2,c)=1}
 \chi_{8\ell}(n_1 n_2) h(c^2 \ell, n_1, n_2).
\end{equation}
Consider the sum over $\ell$, $n_1$, and $n_2$ in \eqref{eq:S21}.
We apply Mellin transforms in the variables $n_1$ and $n_2$ to
see that this sum is
\begin{equation}
\label{eq:S22}
\frac{1}{(2\pi i)^2} \int_{(1+\varepsilon)} \int_{(1+\varepsilon)}
\sumstar_{\ell} {\check h}(c^2 \ell; u, v) \sum_{\substack{n_1, n_2 \\ (n_1n_2, c)=1}}
\frac{\chi_{8\ell}(n_1) \chi_{8\ell}(n_2) }{n_1^{u}n_2^{v} } du \, dv,
\end{equation}
where
\begin{equation*}
{\check h}(x;u,v) = \int_0^{\infty} \int_0^{\infty} h(x,y,z) y^u z^v \frac{dy}{y} \frac{dz}{z}.
\end{equation*}

  Integrating by parts we find that for Re$(u)$, Re$(v) >0$ and any integers $A_i \geq 0, 1 \leq i \leq 3$, we have
\begin{equation}
\label{eq:3.12}
{\check h}(x;u,v) \ll \left( 1 + \frac{x}{X} \right)^{-A_1} \frac{Y^{\text{\rm Re}(u)+\text{\rm Re}(v)}}{|uv|(1+|u|)^{A_2} (1 + |v|)^{A_3}}.
\end{equation}
The sum over $n_1$ and $n_2$ in \eqref{eq:S22} equals $L_c(u,\chi_{8\ell}) L_c(v, \chi_{8\ell})$ where $L_c$ is given by the Euler product defining $L(s,\chi_{8\ell})$
but omitting those primes dividing $c$.   We can then
move the lines of integration in \eqref{eq:S22}
to $\text{Re}(u)=\text{Re}(v)=1/2+1/\log X$ by noting that the Dirichlet $L$-functions have no poles. Then using \eqref{eq:3.12} with $A_2=A_3=1$ and $A_1$ large enough,  together
with
$$
|L_c(u, \chi_{8\ell})L_c(v,\chi_{8\ell})|
\le d(c)^2 ( |L(u, \chi_{8\ell})|^2 + |L(v, \chi_{8\ell})|^2),
$$
we conclude that \eqref{eq:S22} is bounded by
\begin{equation*}
d(c)^2 Y \int_{-\infty}^{\infty} (1+|t|)^{-2} \sumstar_{\ell} \left(1 + \frac{\ell c^2}{X}\right)^{-A_1}
|L(\tfrac 12+\tfrac 1{\log X} +it,  \chi_{8\ell})|^2  \ dt.
\end{equation*}
Now using Corollary \ref{eq:H-B} we conclude that the quantity in \eqref{eq:S22} is
$\ll d(c)^2 X^{1+\varepsilon}Y/c^2$, and using this estimate in \eqref{eq:S21} we obtain that
\begin{align}
\label{S2}
S_2(h) \ll X^{1 + \varepsilon}Y Z^{-1}.
\end{align}

\subsection{Estimating $S_1(h)$: the first main term}
  We evaluate $S_1(h)$ starting from this section. Letting $C = \cos$ and $S = \sin$. By applying Lemma \ref{lem2}, the Poisson summation formula, we obtain
\begin{align}
\label{eq:S1}
 S_1(h)= \frac{X}{2} \sum_{\substack{a \leq Z \\ (a,2)=1}} \frac{\mu(a)}{a^2} \sum_{k \in \mz}
\sum_{(n_1,2a)=1} \sum_{(n_2,2a)=1} \frac{(-1)^kG_k(n_1 n_2)}{n_1 n_2}\int_0^{\infty} h(xX, n_1, n_2) (C + S)\leg{2\pi k xX}{2n_1 n_2 a^2} dx.
\end{align}

   The first main contribution to $S_1(h)$ comes from the $k=0$ term in \eqref{eq:S1}, which
we call $S_{10}(h)$.
Note $G_0(m) = \phi(m)$ if $m = \square$ (a square), and is zero otherwise.  Also note that
$$
\sum_{\substack{a \leq Z \\ (a,2n_1n_2)=1}} \frac{\mu(a)}{a^2} = \frac{1}{\zeta(2)}
\prod_{p|2n_1n_2} \left( 1-\frac{1}{p^2}\right)^{-1} +O(Z^{-1}) = \frac{8}{\pi^2}\prod_{p|n_1n_2} \left(1-\frac{1}{p^2}\right)^{-1}+ O(Z^{-1}) .
$$
 We thus deduce via setting $h_1(y,z) = \int_0^{\infty} h(xX,y,z) dx$ that
\begin{align*}
 S_{10}(h) =  \frac{4X}{\pi^2}
\sum_{\substack{(n_1 n_2,2)=1 \\ n_1 n_2 = \square}} \prod_{p|n_1n_2} \left( \frac{p}{p+1}\right)
h_1\left( n_1, n_2\right)
 + O \Big( \frac XZ \sum_{\substack{(n_1 n_2,2)=1 \\ n_1 n_2 = \square}}
 |h_1(n_1,n_2)|\Big).
\end{align*}
 From our definition of $h$, it is readily seen that $h_1 \ll 1$ and $h_1=0$ unless both $n_1$ and $n_2$ are $\leq Y$. We then deduce that
\begin{align*}
 \sum_{\substack{(n_1 n_2,2)=1 \\ n_1 n_2 = \square}}
 |h_1(n_1,n_2)| \ll  \sum_{n \leq Y}d(n^2) \ll Y\log^2 Y,
\end{align*}
  where the last estimation above follows from \cite{Hooley}.

   It follows that
\begin{equation*}
S_{10}(h) =  \frac{4X}{\pi^2}
\sum_{\substack{(n_1 n_2,2)=1 \\ n_1 n_2 = \square}}  \prod_{p|n_1n_2} \left( \frac{p}{p+1}\right)
h_1\left( n_1, n_2\right) + O\left( \frac {XY\log^2 Y}Z \right).
\end{equation*}

  To analyze the first term above, we use Mellin transforms to see that
$$
h_1(n_1,n_2)  =
\frac{1}{(2\pi i)^2} \int_{(1)} \int_{(1)} \frac{Y^u Y^v}{n_1^u n_2^v}\widetilde{h}_1(u,v) du \, dv,
$$
where
\begin{equation*}
\widetilde{h}_1(u,v) = \int_{\mr_{+}^{3}} h_1(yY,zY) y^{u} z^{v} \frac{dy}{y} \frac{ dz}{ z}.
\end{equation*}

   Similar to \eqref{eq:3.12}, we have that for Re$(u)$, Re$(v) >0$ and any integers $B_i \geq 0, 1 \leq i \leq 2$,
\begin{equation}
\label{h1bound}
 \widetilde{h}_1(u,v) \ll  \frac{1}{|uv|(1+|u|)^{B_1} (1 + |v|)^{B_2}}.
\end{equation}

 By setting
$$
Z(u,v)  = \sum_{\substack{(n_1 n_2,2)=1 \\ n_1 n_2 = \square}} \frac{1}{n_1^{u} n_2^{v}}\prod_{p|n_1 n_2} \left( \frac{p}{p+1}\right),
$$
we see that
\begin{equation}
\label{eqn:4.4}
S_{10}(h)= \frac{4X}{\pi^2} \frac{1}{(2\pi i)^2} \int_{(1)} \int_{(1)} Y^u Y^v \widetilde{h}_1(u,v) Z(u,v)
dv \ du + O\left( \frac {XY\log^2 Y}Z  \right).
\end{equation}

 A simple calculation shows that $Z(u,v)$ equals
 \begin{align*}
  & \prod_{p >2} \left( 1+ \frac p{p+1} \sum_{\substack{(n_1,n_2) \neq (0,0) \\ n_1+n_2 \equiv 0 \pmod 2}} \frac 1{p^{n_1u+n_2v}} \right) \\
  =& \prod_{p >2} \left( 1+ \frac p{p+1} \sum_{\substack{n_1 \text{odd} \\ n_1 \geq 1}}\sum_{\substack{n_2 \text{odd} \\ n_2 \geq 1}} \frac 1{p^{n_1u+n_2v}}+ \frac p{p+1} \sum_{\substack{n_1 \text{even} \\ n_1 \geq 0}}\sum_{\substack{n_2 \text{even} \\ n_2 \geq 0}} \frac 1{p^{n_1u+n_2v}}-\frac p{p+1} \right) \nonumber \\
  =& \prod_{p >2} \left( 1+ \frac p{p+1} \frac 1{p^{u+v}}(1-\frac 1{p^{2u}})^{-1}(1-\frac 1{p^{2v}})^{-1}+ \frac p{p+1} (1-\frac 1{p^{2u}})^{-1}(1-\frac 1{p^{2v}})^{-1}-\frac p{p+1} \right).  \nonumber
 \end{align*}
The Euler product above converges absolutely when Re$(u)$ and Re$(v)$ are both $\geq \frac 12+\epsilon$ for any $\epsilon>0$.
We write
\begin{equation*}
  Z(u,v)=\zeta(2u)\zeta(2v)\zeta(u+v)  Z_2(u,v),
\end{equation*}
where $Z_2(u,v)$ converges absolutely in the region Re$(u)$ and Re$(v)$ larger than $\frac 14+\varepsilon$ and is uniformly bounded there.

 We now use these observations to evaluate the double integral in \eqref{eqn:4.4}.  First
 we move the line of integration in $v$ to Re$(v) = \frac 14+\epsilon$.
 In doing so we encounter a simple pole at $v=1/2$ whose residue contributes
 $$
 \frac{2XY^{1/2}}{\pi^2} \frac{1}{2\pi i} \int_{(1)}  Y^u  \widetilde{h}_1(u,\frac 12) \zeta(2u)\zeta(u+\frac 12)  Z_2(u,\frac 12) du.
 $$
  We now move the line of the above integration in $u$ to Re$(u)=\frac 14+\epsilon$, encountering
a double pole at $u=\frac 12$, and the contribution of the residue of the double pole at $u=0$ is easily seen to be
$$
\frac{XY\log Y}{\pi^2} \widetilde{h}_1(\frac 12,\frac 12)  Z_2(\frac 12,\frac 12) + O(XY).
$$

  To estimate the integral on the $\frac 14+\epsilon$ line, we apply the functional equation for the Riemann Zeta function (\cite[\S 8]{Da}) and Stirling's formula, together with the convex bound for $\zeta(s)$ to see that
\begin{align}
\label{zetabound}
   \zeta(s) \ll \begin{cases}
   1 \qquad & \text{Re$(s) >1$},\\
   (1+|s|)^{\frac {1-\text{Re}(s)}{2}} \qquad & 0< \text{Re$(s) <1$},\\
    (1+|s|)^{\frac 12-\text{Re}(s)} \qquad & \text{Re$(s) \leq 0$}.
\end{cases}
\end{align}

   Applying this and \eqref{h1bound} with $B_1=1, B_2=0$ implies that the integral on the $\frac 14+\epsilon$ line contributes $\ll XY^{\frac 34+\varepsilon}$.

 Lastly, we consider the contribution of
the integral on the $\frac 14+\epsilon$ line of $v$, namely
$$
\frac{4X}{\pi^2} \frac{1}{(2\pi i)^2}  \int_{(\frac 14+\epsilon)} \int_{(1)} Y^u Y^v \widetilde{h}_1(u,v) \zeta(2u)\zeta(2v)\zeta(u+v)  Z_2(u,v)
du \ dv.
$$
 We move the line of integration in $u$ to Re$(u) = \frac 14+\epsilon$.
 In doing so we encounter simple poles at $u=1/2$ and $u=1-v$.

  Keeping \eqref{zetabound} in mind, we now estimate contributions from the poles and integrations in the above process, using \eqref{h1bound}. By taking $B_1=0, B_2=1$ in \eqref{h1bound}, we see that the residue at $u=1/2$ contributes $\ll XY^{\frac 34+\varepsilon}$. By taking $B_1=B_2=0$ in \eqref{h1bound}, so that $\widetilde{h}_1(1-v,v) \ll (|1-v||v|)^{-1}$,  we see that the residues at $u=1-v$ contribute $\ll XY$. By taking $B_1=B_2=1$ in \eqref{h1bound}, we see that the integral on the $\frac 14+\epsilon$ line contributes $\ll XY^{\frac 12+2\varepsilon}$.

Using these estimations in \eqref{eqn:4.4},
we obtain that
\begin{align}
\label{Aequals}
 S_{10}(h) = \frac{XY\log Y}{\pi^2} \widetilde{h}_1(\frac 12,\frac 12)  Z_2(\frac 12,\frac 12)  + O\left( XY+\frac {XY\log^2 Y}Z \right).
\end{align}

\subsection{Estimating $S_1(h)$: the $k \neq 0$ terms}
\label{section:3.3}
We now estimate the contribution to $S_1(h)$ from the terms  $k \neq 0$ in \eqref{eq:S1} and we call
this contribution $S_3(h)$. For any smooth function $f$ on $\mr_+$ with
rapid decay at infinity such that $f$ and all its derivatives have a finite limit as $x\to 0^+$, we consider the Fourier-like transform
\begin{equation*}
 \widehat{f}_{CS}(y) := \int_0^{\infty} f(x) CS(2\pi xy) dx,
\end{equation*}
where $CS$ stands for either $\cos$ or $\sin$.  It is shown in \cite[Sec. 3.3]{S&Y} that
\begin{equation*}
 \widehat{f}_{CS}(y)  = \frac{1}{2\pi i} \int_{(\half)} \widetilde{f}(1-s) \Gamma(s) CS\left(\frac{\text{sgn}(y) \pi s}{2}\right) (2\pi |y|)^{-s} ds.
\end{equation*}

Applying this formula, we have 
\begin{align}
\label{eq:3.30}
\begin{split}
 & \int_0^{\infty} h \left(Xx, n_1, n_2 \right) (C + S)\leg{2\pi k xX}{2n_1 n_2 a^2} dx
\\
= & \frac{X^{-1}}{2\pi i}  \int_{(\varepsilon)} \check{h}\left(1-s; n_1, n_2 \right) \leg{n_1 n_2 a^2}{\pi |k| }^{s} \Gamma(s) (C + \text{sgn}(k)S)\left(\frac{\pi s}{2} \right) ds,
\end{split}
\end{align}
where
\begin{equation*}
\label{eq:3.31}
 \check{h}(s;y,z) = \int_0^{\infty} h(x,y,z) x^s \frac{dx}{x}.
\end{equation*}

Taking the Mellin transforms in the other variables on the second line of \eqref{eq:3.30}, we get
\begin{equation*}
\frac 1X \leg{1}{2\pi i}^3  \int_{(1)} \int_{(1)} \int_{(\varepsilon)} \widetilde{h}\left(1-s,u,v \right) \frac{1}{n_1^{u} n_2^{v}} \leg{n_1 n_2 a^2}{\pi |k| }^{s} \Gamma(s) (C + \text{sgn}(k)S)\left(\frac{\pi s}{2} \right) ds\, du \, dv,
\end{equation*}
  where
\begin{equation*}
\widetilde{h}(s,u,v) = \int_{\mr_{+}^{3}} h(x,y,z) x^{s} y^{u} z^{v} \frac{dx}{x} \frac{dy}{y} \frac{ dz}{ z}.
\end{equation*}

  Integrating by parts we find that for Re$(s)$, Re$(u)$, Re$(v) >0$ and any integers $E_i \geq 0, 1 \leq i \leq 3$,
\begin{equation}
\label{eq:h}
|\widetilde {h}(s,u,v)| \ll \frac{X^{\text{\rm Re}(s)} Y^{\text{\rm Re}(u)+\text{\rm Re}(v)}}{|uvs| (1+|s|)^{E_1} (1+|u|)^{E_2} (1 + |v|)^{E_3}}.
\end{equation}

Using this expression in \eqref{eq:S1}, with the observation that $G_k(m) =G_{4k}(m)$ for odd $m$,
we see that
\begin{multline*}
 S_3(h) = \frac{1}{2} \sum_{\substack{a \leq Z \\ (a,2)=1}} \frac{\mu(a)}{a^2} \sum_{k \neq 0}
\sum_{(n_1,2a)=1} \sum_{(n_2,2a)=1}
\frac{(-1)^kG_{4k}(n_1 n_2)}{n_1 n_2}
\\
\leg{1}{2\pi i}^3  \int_{(1)} \int_{(1)} \int_{(\varepsilon)} \widetilde{h}\left(1-s,u,v \right) \frac{1}{n_1^{u} n_2^{v}} \leg{n_1 n_2 a^2}{\pi |k|  }^{s} \Gamma(s) (C + \text{sgn}(k)S)\left(\frac{\pi s}{2} \right) ds \, du \, dv.
\end{multline*}

  Suppose $\epsilon \in \{ \pm \}$ is the sign of $k$.  Then we have $S_3(h)=S_3^{+}(h) + S_3^{-}(h)$, where
\begin{multline*}
 S_3^{\epsilon}(h) = \frac{1}{2} \sum_{\substack{a \leq Z \\ (a,2)=1}} \frac{\mu(a)}{a^2} \leg{1}{2\pi i}^3  \int_{(1)} \int_{(1)} \int_{(\varepsilon)} \widetilde{h}\left(1-s,u,v \right) \\
\times \sum_{(n_1,2a)=1} \sum_{(n_2,2a)=1}\frac{1}{n_1^{1+u} n_2^{1+v}} \leg{n_1 n_2 a^2}{\pi  }^{s} \Gamma(s) (C + \epsilon S)\left(\frac{\pi s}{2} \right)
\sum_{k \geq 1}\frac{(-1)^kG_{\epsilon 4k}(n_1 n_2)}{k^{s}}   ds \, du \, dv.
\end{multline*}
We write $4k=k_1k_2^2$ where $k_1$ is a fundamental discriminant, and $k_2$ is
positive, so that the sum over $k$ above is a sum over $k_1$ and $k_2$.  We now apply formula \cite[(5.15)]{Young2} by identifying $f(k)=G_{\epsilon 4k}(n_1 n_2)k^{-s}$ in our case such that $f(4k)=4^{-s}f(k)$ to conclude that
\begin{align}
\label{S3e}
 S_3^{\epsilon}(h) = \frac{1}{2} \sum_{\substack{a \leq Z \\ (a,2)=1}} \frac{\mu(a)}{a^2}\left (\sumflat_{k_1 \text{ odd}} \mathcal{M}_{1, \epsilon}(s,u, v, k_1,a) + \sumflat_{k_1 \text{ even}} \mathcal{M}_{2, \epsilon}(s,u, v, k_1,a) \right ),
\end{align}
  where we use $\sumflat$ to denote a sum over fundamental discriminants and
\begin{multline}
\label{eq:mathcalM1}
 \mathcal{M}_{1, \epsilon}(s,u, v, k_1, a)=
 \leg{1}{2\pi i}^3  \int_{(1)} \int_{(1)} \int_{(\varepsilon)} \widetilde{h}\left(1-s,u,v \right) \\
\times \sum_{(n_1,2a)=1} \sum_{(n_2,2a)=1}\frac{2^{1-2s}-1}{n_1^{1+u} n_2^{1+v}} \leg{n_1 n_2 a^2}{\pi k_1  }^{s} \Gamma(s) (C + \epsilon S)\left(\frac{\pi s}{2} \right)
\sum_{k_2 \geq 1}\frac{G_{\epsilon k_1 k_2^2}(n_1 n_2)}{k^{2s}_2}   ds \, du \, dv.
\end{multline}
  The formula for $\mathcal{M}_{2, \epsilon}(s,u,k_1,l)$ is identical to \eqref{eq:mathcalM1} except that the factor $2^{1-2s} - 1$ is omitted.

  Note that the integral over $s$ in \eqref{eq:mathcalM1} may be taken over any vertical
lines with real part between $0$ and $1$ and the integrals over $u,v$ in \eqref{eq:mathcalM1} may be taken over any vertical
lines with real part large than $1$.  Therefore taking the integrals in \eqref{eq:mathcalM1} to be on the lines
Re$(s)= \frac 12+ \varepsilon$ and ${\text {Re}}(u)={\text {Re}}(v) = 1+2\varepsilon$
we find that
\begin{multline*}
\mathcal{M}_{1, \epsilon}(s,u, v, k_1, a)=  \left(\frac{1}{2\pi i}\right)^3  \int_{(1+2\varepsilon)} \int_{(1+2\varepsilon)}
\int_{(\frac 12+\varepsilon)} {\widetilde h}(1-s,u,v) (2^{1-2s}-1) \Gamma(s)
\\
\hskip 1 in \times   (C+ \epsilon S) \left( \frac{\pi s}{2}\right)
 \left(\frac{a^2}{\pi k_1}\right)^s
Z_{\epsilon}(u-s, v-s, s;a,k_1) ds \, du \, dv,
\end{multline*}
   where $Z_{\epsilon}$ is defined in \eqref{eq:Z}.

Changing variables we conclude that
\begin{multline*}
\mathcal{M}_{1, \epsilon}(s,u, v, k_1, a)= \left(\frac{1}{2\pi i}\right)^3  \int_{(\frac 12+\varepsilon)} \int_{(\frac 12+\varepsilon)}
\int_{(\frac 12+\varepsilon)} {\widetilde h}(1-s,u+s,v+s) (2^{1-2s}-1) \Gamma(s)
\\
\hskip 1in \times   (C+ \epsilon S) \left( \frac{\pi s}{2}\right)
 \left(\frac{a^2}{\pi k_1}\right)^s
Z_{\epsilon}(u, v, s;a,k_1) ds \, du \, dv.
\end{multline*}

We now return to the evaluation of \eqref{S3e}. We apply Lemma \ref{lemma:Z} to write $Z_{\epsilon}(u, v, s;a,k_1)$ as $L(\frac 12+u, \chi_{\epsilon k_1})L(\frac 12+v,\chi_{\epsilon k_1})Z_{2, \epsilon}(u,v,s;a, k_1)$.  We first move the line of integration over $u$ to Re$(u)=\epsilon$. By doing so, we cross a pole of the Dirichlet $L$-functions at $u =\frac 12$ for $\epsilon= k_1 = 1$ only. We denote the possible residue by $R$ and the remaining integrals by $I$. To treat $R$, we further move the line of integration over $v$ to Re$(v)=\epsilon$. By doing so, we cross another pole of the Dirichlet $L$-functions at $v =\frac 12$ for $\epsilon= k_1 = 1$ only.  We denote the possible residue by $R_1$ and the remaining integrals by $I_1$.  Similarly, to treat $I$, we move the line of integration over $v$ to Re$(v)=\epsilon$. By doing so, we cross another pole of the Dirichlet $L$-functions at $v =\frac 12$ for $\epsilon= k_1 = 1$ only.  We denote the possible residue by $R_2$ and the remaining integrals by $I_2$.

   We treat $R_1$ first. By Lemma \ref{lemma:Z}, we have
\begin{align*}
R_1 = & \frac 12 \sum_{\substack{a\le Z \\ (a,2)=1}} \frac{\mu(a)}{a^{2}}  \\
& \times \leg{1}{2\pi i} \int_{(\frac 12+\epsilon)}{\widetilde h}(1-s,\frac 12+s,\frac 12+s) (2^{1-2s}-1) \Gamma(s)  (C+  S) \left( \frac{\pi s}{2}\right)
 \left(\frac{a^2}{\pi }\right)^s Z_{2, 1}(\frac 12,\frac 12,s;a, 1)ds . \nonumber
\end{align*}
   Now we apply Lemma \ref{lemma:Z2} to write $Z_{2, 1}(\frac 12,\frac 12,s;a, 1)$ as $\zeta(2s)\zeta(1+2s)Z_3(s;a)$ and we further move the line of integration in the above expression to Re$(s)=-\frac 18+\epsilon$. By doing so, we cross a simple pole at $s =\frac 12$ and a double pole at $s=0$ due to the simple poles of $\zeta(1+2s)$ and $\Gamma(s)$. We denote the residues by $R_{null}$ and $R_0$, respectively. We also denote the remaining integral by $I_3$.
   It is easy to see that $R_{null}=0$ due to the presence of the factor $2^{1-2s}-1$. To compute $R_0$, we note that the residues of $\zeta(1+2s)$ and $\Gamma(s)$ at $s=0$ are $1/2$ and $1$, respectively (see \cite[\S 10 (2)]{Da} for the residue of $\Gamma(s)$ at $s=0$).  Keeping in mind that $Z^{(i)}_3(0;a) \ll a^{\epsilon}, 0 \leq i \leq 2$, we deduce that
\begin{align}
\label{R0}
\begin{split}
 R_0=& \frac {XY}4\log (\frac {Y^2}{X})\zeta(0)\left (\int_{\mr_{+}^{3}} h(xX,yY,zY)y^{-1/2}z^{-1/2} dx dy dz \right ) \sum_{\substack{a\le Z \\ (a,2)=1}} \frac{\mu(a)Z_3(0;a)}{a^{2}} +O(XY) \\
  =& \frac {XY}4\log (\frac {Y^2}{X})\zeta(0)\widetilde{h}_1(\frac 12,\frac 12) \sum_{(a,2)=1} \frac{\mu(a)Z_3(0;a)}{a^{2}} +O(\frac {XY\log (\frac {Y^2}{X})}{Z^{1-\epsilon}}+XY).
\end{split}
\end{align}

   To estimate $I_3$, we apply \eqref{zetabound} and \eqref{eq:h} with $E_1=E_2=E_3=0$, the observation that $Z_3(s;a) \ll a^{\epsilon}$ for any $\epsilon>0$ by Lemma \ref{lemma:Z2} and the bound
\begin{align}
\label{gammabound}
  |\Gamma(s) (C+\epsilon S)(\pi s/2)| \ll |s|^{\text{Re}(s)-\frac 12}
\end{align}
  to see that
\begin{align}
\label{I3}
  I_3 \ll XY(\frac {Y^2}{X} )^{-\frac 18+\epsilon}\int_{(-\frac 18+\epsilon)}\frac{ |s|^{\text{Re}(s)-\frac 12}(1+|s|)^{1/2-\text{Re}(2s)}(1+|s|)^{(1-\text{Re}(1+2s))/2}ds}{|1-s||\frac 12+s||\frac 12+s|} \ll XY(\frac {Y^2}{X} )^{-\frac 18+\epsilon}.
\end{align}

   Next, we treat $R_2$. By Lemma \ref{lemma:Z} again, we see that the contribution of $R_2$ is
\begin{align*}
& \frac 12 \sum_{\substack{a\le Z \\ (a,2)=1}} \frac{\mu(a)}{a^{2}} \int_{(\epsilon)}\int_{(\frac 12+\epsilon)}{\widetilde h}(1-s,u+s,\frac 12+s) (2^{1-2s}-1) \Gamma(s)  (C+  S) \left( \frac{\pi s}{2}\right)
 \left(\frac{a^2}{\pi }\right)^s \zeta(\frac 12+u) Z_{2, 1}(u,\frac 12,s;a, 1)ds du.
\end{align*}
  Writing $Z_{2, 1}(u,\frac 12,s;a, 1)=\zeta(2s)\zeta(2(s+u))Z_{4}(s; u, a)$ by Lemma \ref{lemma:Z2}, we now move the line of integration over $s$ in the above expression to Re$(s)=\frac 14+\epsilon$. By doing so, we cross simple poles at $s=\frac 12$ and $s =\frac 12-u$.  We apply \eqref{zetabound} and \eqref{eq:h} with $E_1=E_3=0, E_2=1$, the observation that $Z_{4}(u; s, a)  \ll a^{\epsilon}$ for any $\epsilon>0$ by Lemma \ref{lemma:Z2} and \eqref{gammabound} to see that the residue at $s=\frac 12$ contributes
\begin{align}
\label{R2res1}
  \ll XY^{1+\epsilon} \left ( \frac {Y}{X}\right )^{1/2}Z^{\epsilon}\int_{(\epsilon)}\frac{du}{|\frac 12+u||1+u|}
  \ll  XY^{1+\epsilon} \left ( \frac {Y}{X}\right )^{1/2}Z^{\epsilon}.
\end{align}

  Similarly, we apply \eqref{zetabound} and \eqref{eq:h} with $E_1=E_2=E_3=0$, the observation that $Z_{4}(u; s, a)  \ll a^{\epsilon}$ for any $\epsilon>0$ by Lemma \ref{lemma:Z2} and \eqref{gammabound} to see that the residue at $s=\frac 12$ contributes
\begin{align}
\label{R2res}
  \ll XY \left ( \frac {Y}{X}\right )^{1/2-\epsilon}\int_{(\epsilon)}\frac{ |1/2-u|^{-\text{Re}(u)}(1+|u|)^{1/4+\epsilon}du}{|\frac 12+u||1-u|}
  \ll XY \left ( \frac {Y}{X}\right )^{1/2-\epsilon}.
\end{align}

  Now, we treat the remaining integral at Re$(u)=\epsilon$ and Re$(s)=\frac 14+\epsilon$. We see that this is
\begin{align*}
& \ll \sum_{a\le Z} \frac{1}{a^{2}} \int_{(\frac 14+\epsilon)}\int_{(\epsilon)}{\widetilde h}(1-s,u+s,\frac 12+s) (2^{1-2s}-1) \Gamma(s)  (C+  S) \left( \frac{\pi s}{2}\right)
 \left(\frac{a^2}{\pi }\right)^s \zeta(\frac 12+u) Z_{2, 1}(u,\frac 12,s;a, 1)du ds \\
& \ll  \sum_{a\le Z} \frac{1}{a^{2}}  \int_{(\frac 14+\epsilon)} \int_{(\frac 14+2\epsilon)}{\widetilde h}(1-s,u,\frac 12+s) (2^{1-2s}-1) \Gamma(s)  (C+  S) \left( \frac{\pi s}{2}\right)
 \left(\frac{a^2}{\pi }\right)^s \zeta(\frac 12+u-s) Z_{2, 1}(u-s,\frac 12,s;a, 1)du ds  \\
& \ll  \sum_{a\le Z} \frac{1}{a^{2}}   \int_{(\frac 14+\epsilon)}\int_{(\frac 14+2\epsilon)}{\widetilde h}(1-s,u,\frac 12+s) (2^{1-2s}-1) \Gamma(s)  (C+  S) \left( \frac{\pi s}{2}\right)
 \left(\frac{a^2}{\pi }\right)^s \zeta(\frac 12+u-s) \zeta(2u)\zeta(2s)Z_{4}(s;u-s, a)du ds.
\end{align*}
   We apply \eqref{zetabound} and \eqref{eq:h} with $E_1=E_3=0, E_2=1$, the observation that $Z_{4}(s; u-s, a) \ll a^{\epsilon}$ for any $\epsilon>0$ by Lemma \ref{lemma:Z2} and \eqref{gammabound} to see that the last expression above is
\begin{align}
\label{R2int}
\begin{split}
  \ll  X^{3/4-\epsilon}Y^{1+3\epsilon} \int_{(\frac 14+\epsilon)}\int_{(\frac 14+2\epsilon)}\frac{|s|^{\text{Re}(s)-\frac 12}(1+|u-s|)^{1/4+\epsilon}(1+|u|)^{1/4+\epsilon}(1+|s|)^{1/4+\epsilon}du ds}{|1-s||\frac 12+s||u||1+u|}  \ll X^{3/4-\epsilon}Y^{1+3\epsilon} .
\end{split}
\end{align}

   Combining \eqref{R2res1}, \eqref{R2res} and \eqref{R2int}, we see that
\begin{align}
\label{R2}
  R_2 \ll  XY^{1+\epsilon} \left ( \frac {Y}{X}\right )^{1/2}Z^{\epsilon}+XY \left ( \frac {Y}{X}\right )^{1/2-\epsilon}+X^{3/4-\epsilon}Y^{1+3\epsilon}.
\end{align}

   By symmetry, we conclude that we also have
\begin{align}
\label{I1}
  I_1 \ll   XY^{1+\epsilon} \left ( \frac {Y}{X}\right )^{1/2}Z^{\epsilon}+XY \left ( \frac {Y}{X}\right )^{1/2-\epsilon}+X^{3/4-\epsilon}Y^{1+3\epsilon}.
\end{align}

   Lastly, we estimate the contribution of $I_2$ to \eqref{S3e}. We split that sum over $k_1$ into
two terms based on whether $k_1 \le K$ or not, for a suitable $K$ to be chosen later.
For the first category of terms we move the lines of integration to Re$(s)= c_1$ for some $1/2<c_1<1$,
Re$(u)=\text{Re}(v)=\epsilon$, and for the second category we move
the lines of integration to Re$(s)=c_2$ for some $c_2>1$, Re$(u)=\text{Re}(v)=\epsilon$.
we find by Lemma \ref{lemma:Z} that for any $\epsilon>0$,
$$
Z_{\epsilon}(u, v,s;a,k_1)
\ll (ak_1)^{\epsilon}|L(\tfrac 12+u, \chi_{\epsilon k_1}) L(\tfrac 12+v, \chi_{\epsilon k_1})|
$$
which is
\begin{equation}
\label{eq:Zbound}
\ll (ak_1)^{\epsilon}  \left(|L(\tfrac 12+u, \chi_{\epsilon k_1})|^2 +
|L(\tfrac 12+v, \chi_{\epsilon k_1})|^2\right).
\end{equation}

Using \eqref{eq:h} with $E_1=E_2=E_3=1$ and \eqref{gammabound}, together with \eqref{eq:Zbound} by noting the symmetry in $u$ and $v$,
we find that our first category of terms
contributes
\begin{multline}
\label{eq:firstbd}
\ll X^{1-c_1} Y^{2c_1+2\epsilon} \sum_{a\le Z} \frac{1}{a^{2-2c_1-\epsilon}}   \int_{(c_1)} \int_{(\epsilon)}\int_{(\epsilon)} \\
\times \sumflat_{k_1\le K}\frac{1}{k_1^{c_1-\epsilon}} |L(\frac 12+u, \chi_{\epsilon k_1})|^2 \frac{ |s|^{\text{Re}(s)-\frac 12} du \, dv \, ds}{|1-s|(1+|1-s|)|u+s|(1+|u+s|)|v+s|(1+|v+s|)}.
\end{multline}

  Applying Lemma \ref{eq:H-B} and partial summation, we see that
\begin{align*}
 \sumstar_{k_1\le K}\frac{1}{k_1^{c_1-\epsilon}} |L(\frac 12+u, \chi_{\epsilon k_1})|^2 \ll K^{1-c_1+2\epsilon}(1+|\text{Im}(u)|)^{1/2+\epsilon} \ll K^{1-c_1+2\epsilon}\left ((1+|u+s|)^{1/2+\epsilon}+|s|^{1/2+\epsilon} \right ).
\end{align*}

  Using the above bound in \eqref{eq:firstbd}, we deduce that the first category of terms
contributes
\begin{align}
\label{eq:firstbd1}
\ll X^{1-c_1} Y^{2c_1+2\epsilon} K^{1-c_1+2\epsilon}Z^{2c_1-1+2\epsilon}.
\end{align}

  Similarly, the contribution of the second category of terms is
\begin{align}
\label{eq:secbd}
\ll X^{1-c_2} Y^{2c_2+2\epsilon} K^{1-c_2+2\epsilon}Z^{2c_2-1+2\epsilon}.
\end{align}
  We now take $K=Y^2Z^2/X$ so that
\begin{align*}
 X^{1-c_1} Y^{2c_1} K^{1-c_1}Z^{2c_1-1}= X^{1-c_2} Y^{2c_2} K^{1-c_2}Z^{2c_2-1}.
\end{align*}

  We then deduce from \eqref{eq:firstbd1} and \eqref{eq:secbd} that the sum of the contributions of the two categories of terms is
\begin{align}
\label{I2}
 \ll (YZ)^{\epsilon}Y^2Z.
\end{align}

\subsection{Conclusion}
   We combine \eqref{S2},  \eqref{Aequals}, \eqref{R0}, \eqref{I3}, \eqref{R2}, \eqref{I1} and \eqref{I2} to see that when $Y \leq X \leq Y^2$,
\begin{align}
\label{Sraw}
\begin{split}
S(X,Y;\Phi, W) =& \frac{XY\log Y}{\pi^2} \widetilde{h}_1(\frac 12,\frac 12)  Z_2(\frac 12,\frac 12) +\frac {XY}4\log (\frac {Y^2}{X})\zeta(0)\widetilde{h}_1(\frac 12,\frac 12) \sum_{(a,2)=1} \frac{\mu(a)Z_3(0;a)}{a^{2}} \\
&+O \left (XY+XY^{1+\epsilon} \left ( \frac {Y}{X}\right )^{1/2}Z^{\epsilon}+XY(\frac {Y^2}{X} )^{-\frac 18+\epsilon}+XY \left ( \frac {Y}{X}\right )^{1/2-\epsilon} \right )\\
&+O\left(X^{1 + \varepsilon}YZ^{-(1-\epsilon)}+(YZ)^{\epsilon}Y^2Z \right).
\end{split}
\end{align}

   We now set $Z=(X/Y)^{1/2}$ and define
\begin{align}
\label{C1C2}
\begin{split}
  C_1(\Phi, W) =& \frac{1}{\pi^2} \widetilde{h}_1(\frac 12,\frac 12)  Z_2(\frac 12,\frac 12), \\
  C_2(\Phi, W) =& \frac {1}4\zeta(0)\widetilde{h}_1(\frac 12,\frac 12) \sum_{(a,2)=1} \frac{\mu(a)Z_3(0;a)}{a^{2}}.
\end{split}
\end{align}
   The proof of Theorem \ref{meansquare} then follows from \eqref{Sraw}.

\noindent{\bf Acknowledgments.} The author is supported in part by NSFC grant 11871082.

\bibliography{biblio}
\bibliographystyle{amsxport}



\end{document}